\documentclass{amsart}
\usepackage[utf8]{inputenc}
\usepackage[english]{babel}
\usepackage{amssymb,upref,calc,url,amsmath,amscd,amsthm,verbatim} 
\usepackage[mathscr]{eucal}
\usepackage{graphicx}
\usepackage[all]{xy}
\usepackage[in]{fullpage}

\theoremstyle{theorem}
\newtheorem{thm}{Theorem}
\newtheorem{cor}[thm]{Corollary}
\newtheorem{lem}[thm]{Lemma}
\newtheorem{problem}[thm]{Problem}

\newtheorem{rem}[thm]{Remark}
\theoremstyle{remark}
\newtheorem{ex}[thm]{Example}
\theoremstyle{definition}
\newtheorem{dfn}[thm]{Definition}

\begin{document}

\title{Algebraically integrable quadratic dynamical systems}
\author{Victor M. Buchstaber, Elena Yu. Bunkova}
\date{}

\address{Steklov Institute of Mathematics, Russian Academy of Sciences, ul. Gubkina 8, Moscow, 119991 Russia.}
\email{bunkova@mi.ras.ru (E.Yu.Bunkova), buchstab@mi.ras.ru (V.M.Buchstaber).}

\thanks{The work is supported by RFBR grant 12-01-33058.}

\maketitle

\begin{abstract}

We consider in $\mathbb{C}^n$ with coordinates $\xi = (\xi_1, \dots, \xi_n)$ the class of symmetric homogeneous quadratic dynamical systems. We introduce the notion 
of algebraic integrability for this class.

We present a class of symmetric quadratic dynamical systems that are algebraically integrable by the set of functions $h_1(t), \dots, h_n(t)$ where $h_1(t)$ is any solution of an ordinary differential equation of order $n$ and $h_k(t)$ are differential polynomials in $h_1(t)$, $k = 2, \dots, n$. We describe a method of constructing this ordinary differential equation.
We give a classification of symmetric quadratic dynamical systems and describe the maximal subgroup in $\mathrm{GL}(n, \mathbb{C})$ that acts on this systems.

We apply our results to analysis of classical systems of Lotka-Volterra type and Darboux-Halphen system and their modern generalizations.

\end{abstract}

\section*{Introduction.} \text{}

Our study of polynomial dynamical systems (see \cite{FA}) drew our attention to the class of symmetric quadratic dynamical systems (see definition below).
For such systems we introduce the notion of algebraic integrability and in the generic and almost generic case find solutions of the algebraic integrability problem. 
Our approach reduces the problem of integrability of a symmetric quadratic dynamical system to the question of solving an ordinary differential equation. 
This opens the way to apply ideas of Kovalevskaya and Painlev\'e:
to find solutions of dynamical systems in terms of special functions which are solutions of remarkable ordinary differential equations. The classical example here is the solution of Darboux-Halphen system in terms of Chazy equation solutions.

In the first part of the article we consider for any $n$ the space of quadratic dynamical systems and the canonical action of $\mathrm{GL}(n, \mathbb{C})$ on it. We describe the subspace of symmetric quadratic dynamical systems and the maximal subgroup in $\mathrm{GL}(n, \mathbb{C})$ acting on the space of such systems.
In the second part we introduce the notion of algebraic integrability and give an algorithm of solving the problem of algebraic integrability for generic and almost generic symmetric quadratic dynamical systems. This algorithm reduces the problem of algebraic integrability to the problem of solving an ordinary differential equation. We give the action on such equations of the normalizer subgroup in $\mathrm{GL}(n, \mathbb{C})$ of generic symmetric quadratic dynamical systems. In the third part we show how our method works for known problems of quadratic dynamical systems theory.

The authors are grateful to \`E.~B.~Vinberg for useful discussions.

	\section{Symmetric quadratic dynamical systems.}

	\subsection{Quadratic dynamical systems.} \text{}

Fix some $n \in \mathbb{N}$. Let us consider in $\mathbb{C}^n$ with coordinates $\xi = (\xi_1, \dots, \xi_n)^\top$ the general quadratic dynamical system 
\begin{equation} \label{e1}
\xi_k'(t) = A_k^{i,j} \xi_i(t) \xi_j(t), \qquad A_k^{i,j} = A_k^{j,i} = const, \qquad k = 1, \dots, n.
\end{equation}
Here and below we use the usual Einstein summation convention.
We identify the space of quadratic dynamical systems with the ${1 \over 2} n^2 (n + 1)$-dimensional linear space of tensors
$A = (A_{k}^{i,j})$ such that $A_{k}^{i,j} = A_{k}^{j,i}$.
Note that the system \eqref{e1} is homogeneous with respect to the grading $\deg t = 4$, $\deg \xi_i = - 4$, $\deg A_k^{i,j} = 0$.

The system \eqref{e1} defines a linear operator $\mathcal{L} = \mathcal{L}(A)\colon \mathbb{C}[\xi_1, \dots, \xi_n] \to \mathbb{C}[\xi_1, \dots, \xi_n]$ with $\deg \mathcal{L} = - 4$
as
\[
\mathcal{L} = A_k^{i,j} \xi_i \xi_j {\partial \over \partial \xi_k}.
\]
Note that $P(\xi)$ is an integral of system \eqref{e1} if and only if $\mathcal{L} P(\xi) = 0$.

The group $\mathrm{GL}(n, \mathbb{C})$ acts by linear change of coordinates $\eta = B \xi$ where $B = (B_i^j) \in \mathrm{GL}(n, \mathbb{C})$ on the space of dynamical systems \eqref{e1}.
This action is a representation of $\mathrm{GL}(n, \mathbb{C})$ on the space of tensors $A = (A_{k}^{i,j})$:
\[
A_{k}^{i,j} \mapsto A_{p}^{q,r} B_k^p (B^{-1})_q^i (B^{-1})_r^j.
\]

	\subsection{Symmetric quadratic dynamical systems.} \text{}

In this section we introduce three definitions of a symmetric quadratic dynamical system and prove their equivalence.

The symmetric group $\mathrm{S}_n$ is the group of permutations of n elements. We identify $\mathrm{S}_n$ with the subgroup in $\mathrm{GL}(n, \mathbb{C})$ generated by the matrices obtained by a transposition of two lines in the identity matrix $E$, thus $\mathrm{S}_n$ acts on $\mathbb{C}^n$ by permutation of coordinates $(\xi_1, \dots, \xi_n)$.

\begin{dfn} \label{d1}
A system \eqref{e1} is called \emph{symmetric} if $\mathrm{S}_n$ lies in the stabilizer of this system.
\end{dfn}

\begin{dfn} \label{d2}
A linear operator $L: \mathbb{C}[\xi_1, \dots, \xi_n] \to \mathbb{C}[\xi_1, \dots, \xi_n]$ is called symmetric if 
$T (L P(\xi)) = L T P(\xi)$ for any $T \in \mathrm{S}_n$ and any polynomial $P(\xi)$, where $T P(\xi) = P(T(\xi))$.
\end{dfn}

\begin{dfn} \label{d3}
A system \eqref{e1} is called \emph{symmetric} if $\mathcal{L}(A)$ is symmetric.
\end{dfn}

Let $\mathrm{Sym} \subset \mathbb{C}[\xi_1, \dots, \xi_n]$ denote the ring of symmetric polynomials.
According to classical results, $\mathrm{Sym}$ is isomorphic to a polynomial ring with $n$ generators.
The best known generators of this ring are the Newton polynomials and the elementary symmetric functions.
We denote the Newton polynomials as $p_k = p_k(\xi) = \sum_i \xi_i^k$ and the elementary symmetric functions as $\sigma_k = \sigma_k(\xi) = \sum_{i_1 < i_2 < \dots < i_k} \xi_{i_1} \xi_{i_2} \dots \xi_{i_k}$. For simplicity set $p_0 = n$.
One can pass from one of this sets of generators to another by the classical identities (see \cite{Mac})
\[
p_k = \left| \begin{matrix}
  \sigma_1 & 1            & 0        & 0      & \dots    & 0      \\
2 \sigma_2 & \sigma_1     & 1        & 0      & \dots    & 0      \\
3 \sigma_3 & \sigma_2     & \sigma_1 & 1      &          & 0      \\
  \vdots   & \vdots       &          & \ddots & \ddots   & \vdots \\
 \vdots   & \vdots       &          &        & \ddots       & 1\\
k \sigma_k & \sigma_{k-1} & \dots    &  \dots & \dots      & \sigma_1
\end{matrix} \right|, \qquad
k! \sigma_k = \left| \begin{matrix}
p_1 & 1            & 0        & 0      & \dots    & 0      \\
p_2 & p_1     & 2        & 0      & \dots    & 0      \\
p_3 & p_2     & p_1 & 3      &          & 0      \\
 \vdots  & \vdots       &          & \ddots & \ddots   & \vdots \\
 \vdots  & \vdots       &          &        & \ddots   & (k-1)  \\
p_k & p_{k-1} & \dots    &  \dots & \dots    & p_1
\end{matrix} \right|.
\]

The universal algebraic map 
\begin{equation} \label{S}
S\colon \mathbb{C}^n \to \mathbb{C}^n\colon \xi \mapsto (\sigma_1, \dots \sigma_n)
\end{equation}
is the projection to the orbit space of the action of $\mathrm{S}_n$ on $\mathbb{C}^n$.
We have $\mathrm{Sym} = S^* \mathbb{C}[\sigma_1, \dots, \sigma_n]$.

\begin{dfn} \label{d4}
A system \eqref{e1} is called \emph{symmetric} if $\mathcal{L} P \in \mathrm{Sym}$ for any $P \in \mathrm{Sym}$.
\end{dfn}

\begin{thm}
The definitions \ref{d1}, \ref{d3} and \ref{d4} are equivalent.
\end{thm}

\begin{proof}
For $T \in \mathrm{S}_n$ denote by $T(i)$ the image of $i$ under the permutation $T$. We have 
\[
(T \mathcal{L}) P(\xi) = A_{T(k)}^{T(i),T(j)} \xi_i \xi_j {\partial \over \partial \xi_k} P(\xi),
\]
and 
\[
T \mathcal{L} P(\xi) = (T \mathcal{L}) (T P(\xi)),
\]
thus definition \ref{d1} is equivalent to $T \mathcal{L} = \mathcal{L}$, which is equivalent to definition \ref{d3}.
Let system \eqref{e1} satisfy definition \ref{d3}. Any $T \in \mathrm{S}_n$ does not change symmetric polynomials, moreover, this is a characteristic property of such polynomials. 
 Each permutation $T$ does not change the operator $\mathcal{L}(A)$.
Hence for a symmetric polynomial $P$ the polynomial $\mathcal{L} P$ is symmetric. 

Let system \eqref{e1} satisfy definition \ref{d4}. The polynomial $\mathcal{L} P = A_k^{i,j} \xi_i \xi_j {\partial \over \partial \xi_k} P$ is symmetric for any symmetric polynomial $P$, and thus does not change under the action of $\mathrm{S}_n$. Therefore, we have $A_{T(k)}^{T(i),T(j)} \xi_i \xi_j {\partial \over \partial \xi_k} P = A_k^{i,j} \xi_i \xi_j {\partial \over \partial \xi_k} P$ for any $T \in \mathrm{S}_n$ and any symmetric $P$. We conclude that $\left(A_{T(k)}^{T(i),T(j)} -  A_k^{i,j}\right) \xi_i \xi_j {\partial \over \partial \xi_k} P = 0$ for any symmetric $P$. Now if for a derivation $L$ we have $L P = 0$ for any symmetric $P$, then $S^* L = 0$. But $S$ is a local diffeomorphism, hence $A_{T(k)}^{T(i),T(j)} = A_k^{i,j}$ for any $T \in \mathrm{S}_n$, and we come to definition \ref{d1}.
\end{proof}

\begin{lem} 
Each symmetric quadratic dynamical system has the form
\begin{equation} \label{e1'}
\xi_k'(t) = A_k^{i,j} \xi_i(t) \xi_j(t), \qquad k = 1, \dots, n,
\end{equation}
where $A_k^{i,j} = A_k^{j,i}$ for any $i,j,k$,
$A^{m,m}_m = A^{m',m'}_{m'}$, $A^{m,m}_{n} = A^{m',m'}_{n'}$, $A^{m,n}_{m} = A^{m',n'}_{m'}$ and $A^{m, n}_l = A^{m', n'}_{l'}$ for any $n,m,l,n',m',l'$ such that $n \ne m$, $l \ne m$, $l \ne n$, $n' \ne m'$, $l' \ne m'$, $l' \ne n'$.
\end{lem}

\begin{proof}
Let $T \in \mathrm{S}_n$ be a transposition of $i$ and $j$. In the general quadratic dynamical system \eqref{e1}
according to definition \ref{d1} we have $A^{i,i}_i = A^{j,j}_j$, $A^{i,i}_k = A^{j,j}_k$, $A^{k,l}_i = A^{k,l}_j$, $A^{i,k}_i = A^{j,k}_j$, $A^{k,i}_l = A^{k,j}_l$ for any $k \ne i$, $k \ne j$, $l \ne i$, $l \ne j$. By applying this equalities for all $i$, $j$ we see that the only possibly pairwise different coefficients of the system are $A^{m,m}_m$, $A^{m,m}_n$, $A^{m,n}_m$ and $A^{m, n}_l$ for $n \ne m$, $l \ne m$, $l \ne n$.
\end{proof}

\begin{cor} \label{cor6}
Each symmetric quadratic dynamical system has the form
\begin{equation} \label{e2}
\xi_k' = \alpha \xi_k^2 + \beta \xi_k p_1(\xi) + \gamma p_1(\xi)^2 + \delta p_2(\xi), \qquad k = 1, \dots, n,
\end{equation}
for some coefficients $\alpha$, $\beta$, $\gamma$, $\delta$.
\end{cor}

\begin{proof}
 The form \eqref{e2} is a linear change of variables of \eqref{e1'}, namely
$A^{m,m}_m  = \alpha + \beta + \gamma + \delta$, $A^{m,m}_n = \gamma + \delta$, $2 A^{m,n}_m = \beta + 2 \gamma$, $A^{m,n}_l = \gamma$.
\end{proof}

\begin{lem} \label{lem2}
The stabilizer subgroup in $\mathrm{GL}(n, \mathbb{C})$ of symmetric quadratic dynamical systems coincides with $\mathrm{S}_n$.
\end{lem}

\begin{proof}
Consider the system 
\[
\xi_k' = \xi_k^2, \qquad k = 1, \dots, n.
\]
Each $g$ in the stabilizer subgroup does not change this system. Let $g$ bring $\xi_k$ into $\eta_k =  a^s_{k} \xi_s$. The dynamical system in $\eta_k$ is
\[
\eta_k' = \eta_k^2, \qquad k = 1, \dots, n.
\]
Thus for any $\xi$ we have
\[
a^{i}_{k} \xi_i^2 = (a^i_{k})^2 \xi_i^2 + 2 \sum_{i < j} a^i_{k} a^j_{k} \xi_i \xi_j,
\]
therefore $a^i_{k} \ne 0$ only for one value of $i$ for each $k$ and $a^i_{k} = (a^i_{k})^2$. Using $g \in \mathrm{GL}(n, \mathbb{C})$ we conclude that there exists a permutation $T \in \mathrm{S}_n$ such that $a_{k,i} = 1$ if $i = T(k)$ and $a_{k,i} = 0$ otherwise.
\end{proof}

\begin{dfn}
The system \eqref{e1} is said to be \emph{quasi-symmetric} with respect to $B \in \mathrm{GL}(n, \mathbb{C})$ if this system in the coordinates $\eta = B \xi$ 
is symmetric.
\end{dfn}

	\subsection{Integrals of symmetric quadratic dynamical systems.} \text{}

\begin{lem} \label{lem9}
Let $P(\xi)$ be an integral of a symmetric quadratic dynamical system \eqref{e1'}. Then the polynomials $T(P(\xi)) = P(T(\xi))$ where $T \in \mathrm{S}_n$ are integrals of the same system.
\end{lem}

\begin{proof}
According to definition \ref{d3} we have $\mathcal{L}(T P(\xi)) = T (\mathcal{L} P(\xi)) = 0$, thus $T P(\xi)$ is an integral of the symmetric system \eqref{e1'}.
\end{proof}

\begin{rem}
Each integral $P(\xi)$ of a symmetric quadratic dynamical system \eqref{e1'} can be presented as the sum of its graded components, one component in each grading, each component being an integral of the system.
\end{rem}

Let us consider the projector $\pi: \mathbb{C}[\xi_1, \dots, \xi_n] \to \mathrm{Sym}$ defined as
\[
\pi P(\xi) = {1 \over n!} \sum_{T \in \mathrm{S}_n} T P(\xi).
\]

\begin{cor}
Let $P(\xi)$ be a homogeneous integral of a symmetric quadratic dynamical system \eqref{e1'}. Then the polynomial $\mathcal{P}(\xi) = \pi P(\xi)$ is a homogeneous symmetric integral of the same system.
\end{cor}

\begin{proof}
Obviously, a sum of integrals of a dynamical system is itself an integral of the dynamical system. Therefore the corollary follows from Lemma \ref{lem9}.
\end{proof}

	\subsection{Generic symmetric quadratic dynamical systems.}\text{}

In this section we introduce the notion of generic symmetric quadratic dynamical systems and classify such systems for every $n$.

For homogeneous generators $a_1, \dots, a_n$ of the ring $\mathrm{Sym}$ with $\deg a_k = - 4 k$ a symmetric system \eqref{e1'} implies a homogeneous dynamical system
\begin{equation} \label{e3}
a_k' = \mathcal{L} a_k, \qquad k = 1, \dots, n.
\end{equation}
Using the grading one can rewrite \eqref{e3} in the form
\begin{equation} \label{e4}
a_k' = g_{k+1}(a_1, \dots, a_k) + c_k a_{k+1}, \qquad k = 1, \dots, n, 
\end{equation}
where $c_{n} = 0$ and $g_j(a_1, \dots, a_{j-1})$ are homogeneous polynomials with $\deg g_j = - 4 j$.

\begin{dfn}
A homogeneous symmetric system \eqref{e1'} is \emph{generic} if $c_k \ne 0$ for $k = 1, \dots, n-1$.
\end{dfn}

\begin{lem} \label{lem14}
A system being generic does not depend on the choice of generators in the ring.
\end{lem}

\begin{proof}
Let $a_1, \dots, a_n$ and $b_1, \dots, b_n$ be two generators in the ring with $\deg a_j = \deg b_j = - 4 j$. Consider the ideal $J$ generated by decomposable elements of $\mathrm{Sym}$. We have $a_k = C_k b_k \mod J$, where $C_k \ne 0$, $C_k \in \mathbb{C}$. Thus, \eqref{e4} implies
\[
b_k' = G_{k+1}(b_1, \dots, b_k) + c_k {C_{k+1} \over C_k} b_{k+1}, \qquad k = 1, \dots, n, 
\]
for some homogeneous polynomials $G_j(b_1, \dots, b_k)$. Therefore $c_k \ne 0 \Leftrightarrow c_k {C_{k+1} \over C_k} \ne 0$. 
\end{proof}

\begin{thm}\label{t11}
For $n = 1$ each generic symmetric quadratic dynamical system has the form
\begin{equation} \label{e5}
\xi_1' = \alpha \xi_1^2.
\end{equation}

For $n = 2$ each generic symmetric quadratic dynamical system has the form
\begin{equation} \label{e6}
\xi_k' = \alpha \xi_k^2 + \beta \xi_k p_1(\xi) + \gamma p_1(\xi)^2, \qquad k = 1, 2,
\end{equation}
with $\alpha \ne 0$.

For $n \geqslant 3$ each generic symmetric quadratic dynamical system has the form
\begin{equation} \label{e7}
\xi_k' = \alpha \xi_k^2 + \beta \xi_k p_1(\xi) + \gamma p_1(\xi)^2 + \delta p_2(\xi), \qquad k = 1, \dots, n,
\end{equation}
with $\alpha \ne 0$, $\alpha + n \delta \ne 0$.
\end{thm}

\begin{proof}
For $n = 1$ any quadratic dynamical system has the form \eqref{e5}, it is symmetric and the condition for it being generic gives no restrictions.

For $n = 2$ using corollary \ref{cor6} and the relation $p_2 - p_1^2 + 2 \xi_1 p_1 - 2 \xi_1^2 = 0$ we get the form \eqref{e6}. Therefore
\[
p_1(\xi)' = \left(\beta + 2 \gamma\right) p_1(\xi)^2 + \alpha p_2(\xi),
\]
which is the first equation of \eqref{e4} for $a_k = p_k$. Thus \eqref{e6} is generic for $\alpha \ne 0$.

For $n \geqslant 3$ we get \eqref{e7} from corollary \ref{cor6} directly. Therefore
\[
{1 \over k} p_{k}(\xi)' = \alpha p_{k+1}(\xi) + \beta p_{k}(\xi) p_1(\xi) + \gamma p_{k-1}(\xi) p_1(\xi)^2 + \delta p_{k-1}(\xi) p_2(\xi), \qquad k = 1, 2, \dots,
\]
where $p_0(\xi) = n$. Thus in \eqref{e4} for $a_k = p_k$ we have $c_1 = \alpha + n \delta$ and $c_k = \alpha$ for $k = 2, \dots, n-1$.
\end{proof}

\begin{dfn}
A quasi-symmetric with respect to $B$ system is \emph{generic} if the symmetric system in the coordinates $\eta = B \xi$ is generic.
\end{dfn}

	\subsection{Almost generic symmetric quadratic dynamical systems.}\text{}
\begin{dfn}
A homogeneous symmetric system \eqref{e1'} is \emph{almost generic} if in \eqref{e4} we have $c_k \ne 0$ for $k = 1, \dots, n-2$.
\end{dfn}

\begin{lem}
A system being almost generic does not depend on the choice of generators in the ring.
\end{lem}

The proof coincides with the proof for Lemma \ref{lem14}.

\begin{thm} \label{tas}
For $n = 2$ each symmetric quadratic dynamical system is almost generic.

For $n = 3$ each almost generic symmetric quadratic dynamical system has the form
\eqref{e7} with $\alpha + n \delta \ne 0$.

For $n \geqslant 4$ each almost generic symmetric quadratic dynamical system is generic.
\end{thm}

The proof is the proof of Theorem \ref{t11} with obvious modifications.

\begin{dfn}
A quasi-symmetric with respect to $B$ system is \emph{almost generic} if the symmetric system in the coordinates $\eta = B \xi$ is almost generic.
\end{dfn}

	\subsection{Action of $\mathrm{GL}(n, \mathbb{C})$ on symmetric quadratic dynamical systems.} \text{}

Recall $E$ is the $n \times n$ identity matrix and denote by $e$ the $n$-dimensional vector $e = (1, \dots, 1)^\top$.

Consider the set of matrices $\mathcal{B}(\lambda, q) = \lambda E + q e e^\top$. With respect to the standard multiplication of matrices it is a commutative monoid with multiplication 
\[
\mathcal{B}(\lambda_1, q_1) \mathcal{B}(\lambda_2, q_2) = \mathcal{B}(\lambda_1 \lambda_2, \lambda_1 q_2 + \lambda_2 q_1 + n q_1 q_2).
\]
We have $\det \mathcal{B}(\lambda, q) = \lambda^{n-1} (\lambda + n q)$, therefore $\mathcal{B}(\lambda, q) \in \mathrm{GL}(n, \mathbb{C})$ if $\lambda \ne 0$, $\lambda \ne - n q$. Denote by $\mathcal{B}$ the subgroup of matrices $\mathcal{B}(\lambda, q) \in \mathrm{GL}(n, \mathbb{C})$. The inverse element is 
\[
\mathcal{B}(\lambda, q)^{-1} = \mathcal{B}\left({1 \over \lambda}, - {q \over \lambda (\lambda + n q)}\right).
\]

\begin{lem}\label{lemma17}
The centralizer of $\mathrm{S}_n$ in $\mathrm{GL}(n, \mathbb{C})$ is $\mathcal{B}$.
\end{lem}

\begin{proof}
From the definition of $\mathcal{B}$ we see that $\mathcal{B}$ commutes with $S_n$. Let $G \in \mathrm{GL}(n, \mathbb{C})$ satisfy $T^{-1} G T = G$ for any $T \in S_n$. Then $G_{i}^i = G_{j}^j$, $G_{i}^j = G_{j}^i$ for any $i$, $j$ by the action of $T = (i,j)$, and $G_{i}^j = G_{i}^k$, $G_{j}^i = G_{k}^i$ for any $i \ne j$, $i \ne k$ by the action of $T = (j,k)$. Therefore $G = \lambda E + q e e^\top$ for $\lambda + q = G_{i}^i$ and $q = G_{i}^j$ for any $i \ne j$.
\end{proof}

The following lemma describes the action of $\mathcal{B}$ on symmetric quadratic dynamical systems.

\begin{lem}
The action $\eta = \mathcal{B}(\lambda, q) \xi$ with $\lambda \ne 0$, $\lambda + n q \ne 0$ brings each generic symmetric quadratic dynamical system \eqref{e2}
into the generic symmetric quadratic dynamical system
\[
\eta_k' = {\alpha  \over \lambda} \eta_k^2 - {2 q \alpha - \lambda \beta \over \lambda (\lambda + n q)} \eta_k p_1(\eta) - {q^2 (\alpha  + n \delta) - \lambda (\lambda \gamma - 2 q \delta) \over \lambda^2 (\lambda + n q)} p_1(\eta)^2 + {q (\alpha  + n \delta) + \lambda \delta\over \lambda^2} p_2(\eta).
\]
\end{lem}

\begin{proof}
We have $\eta = \lambda \xi + q p_1(\xi) e$, thus $p_1(\eta) = (\lambda + n q) p_1(\xi)$ and $p_2(\eta) = \lambda^2 p_2(\xi) + (2 \lambda + n q) q p_1(\xi)^2$.
Substituting this relations into \eqref{e2} we get the required system, which is generic for $\alpha \ne 0$ and $\alpha + n \delta \ne 0$.
\end{proof}

\begin{cor} \label{c13}
The action $\eta = \mathcal{B}(\lambda, q) \xi$ with $\lambda \ne 0$, $\lambda + n q \ne 0$ brings each generic symmetric quadratic dynamical system with $n = 1$
into the generic symmetric quadratic dynamical system
\[
\eta_1' = {\alpha \over \lambda + q} \eta_1^2
\]
and with $n=2$ into the generic symmetric quadratic dynamical system
\[
\eta_k' = {(\lambda + 2 q) \alpha \over \lambda^2} \eta_k^2 - {4 q (\lambda + q) \alpha - \lambda^2 \beta \over \lambda^2 (\lambda+2 q)} \eta_k p_1(\eta) + {q (\lambda + q) \alpha + \lambda^2 \gamma \over \lambda^2 (\lambda+2 q)} p_1(\eta)^2, \quad k = 1, 2.
\]
\end{cor}

\begin{lem} \label{eta}
The normalizer subgroup in $\mathrm{GL}(n, \mathbb{C})$ of generic symmetric quadratic dynamical systems is the normalizer of $S_n$ in $\mathrm{GL}(n,\mathbb{C})$.
\end{lem}

\begin{proof}
According to lemma \ref{lem2} the stabilizer subgroup is $\mathrm{S}_n$. Let $h \in \mathrm{S}_n$ and let $g$ lie in the normalizer subgroup in $\mathrm{GL}(n, \mathbb{C})$ of generic symmetric quadratic dynamical systems.
Then $h' = g h g^{-1}$ lies in the stabilizer subgroup $\mathrm{S}_n$.
\end{proof}

\begin{thm}
The normalizer subgroup in $\mathrm{GL}(n, \mathbb{C})$ of generic symmetric quadratic dynamical systems is equal to $\mathcal{B}\times\mathrm{S}_n$.
\end{thm}

\begin{proof}
Let us denote the normalizer subgroup in $\mathrm{GL}(n, \mathbb{C})$ of generic symmetric quadratic dynamical systems by~$G$.

Form Lemma~\ref{eta}
it follows that any element~$g\in G$ induces an automorphism~$\varphi_g(h)=ghg^{-1}$ of the group~$\mathrm{S}_n$. It is well known that for~$n\ne6$ the group~$\mathrm{S}_n$ has no outer automorphisms, thus $\varphi_g$ is an inner automorphism. For the proof that $\varphi_g$ is an inner automorphism for $n=6$ see below.

Take any~$g\in G$. Since~$\varphi_g$ is an inner automorphism, there is an element~$h'\in\mathrm{S}_n$ such that~$\varphi_g=\varphi_{h'}$,~i.e.~$ghg^{-1}=h'h(h')^{-1}$ for any~$h\in\mathrm{S}_n$. This can be rewritten as~$((h')^{-1}g)h=h((h')^{-1}g)$. Therefore $(h')^{-1}g\in\mathrm{GL}(n, \mathbb{C})$ commutes with any~$h\in \mathrm{S}_n$. By Lemma \ref{lemma17} we get $(h')^{-1}g\in\mathcal{B}$.

Thus~$G$ is generated by~$\mathcal{B}$ and~$\mathrm{S}_n$ in~$\mathrm{GL}(n,  \mathbb{C})$. From~$\mathcal{B}\cap\mathrm{S}_n=\{e\}$ and~$bh=hb$ for any~$b\in\mathcal{B}$ and~$h\in\mathrm{S}_n$ we get~$G=\langle\mathcal{B},\mathrm{S}_n\rangle=\mathcal{B}\times\mathrm{S}_n$.

For~$n=6$ let~$\psi$ be a non-inner automorphism. This automorphism takes a transposition~$T_1=(i,j)$ to a product of transpositions~$T_2=(i_1,i_2)(i_3,i_4)(i_5,i_6)$. It follows from the fact that $\psi$ preserves order and sign, but does not preserve cyclic type of permutation as a non-inner automorphism. Consider~$T_1$ and~$T_2$ as elements of~$\mathrm{GL}(n,  \mathbb{C})$. The element~$T_1$ has eigenvalues~$(i,-i,1,1,1,1)$, the element~$T_2$ has eigenvalues~$(i,-i,i,-i,i,-i)$. This implies that~$T_1$ and~$T_2$ are not conjugate in~$\mathrm{GL}(n,  \mathbb{C})$. Hence,~$\varphi_g$ is an inner automorphism of~$\mathrm{S}_n$ as well for~$n=6$ and any~$g\in G$.
\end{proof}

\begin{lem} \label{lem16} The orbits of the action of $\mathcal{B}(\lambda, q)$ on generic symmetric quadratic dynamical systems are as follows:

For $n = 1$ there are two orbits: $\alpha \ne 0$ and $\alpha = 0$ in \eqref{e5}.

For $n = 2$ and $\alpha \ne - \beta$ there is a one-parametric space of orbits with exactly one representative of the form
\begin{equation} \label{e8}
\xi_k' = \xi_k^2 + \widetilde{\gamma} p_1(\xi)^2, \quad k = 1, 2,
\end{equation}
in each orbit,
for $\alpha \ne 4 \gamma$ there is a one-parametric space of orbits with exactly one representative of the form
\begin{equation} \label{e9}
\xi_k' = \xi_k^2 + \widetilde{\beta} \xi_k p_1(\xi), \quad k = 1, 2,
\end{equation}
in each orbit, 
for $\alpha = - \beta = 4 \gamma$ a representative is
\begin{equation} \label{e10}
\xi_k' = \xi_k^2 - \xi_k p_1(\xi) + {1 \over 4} p_1(\xi)^2, \quad k = 1, 2.
\end{equation}

For $n \geqslant 3$ there is a two-parametric space of orbits with exactly one representative of the form
\begin{equation} \label{e11}
\xi_k' = \xi_k^2 + \widetilde{\beta} \xi_k p_1(\xi) + \widetilde{\gamma} p_1(\xi)^2, \quad k = 1, \dots, n,
\end{equation}
in each orbit.
\end{lem}

\begin{proof}
For $n=1$ this result follows from corollary \ref{c13}.

For $n = 2$ for $\alpha \ne - \beta$ the non-degenerate matrix $\mathcal{B}(\lambda, q)$ with $\lambda = \sqrt{\alpha (\alpha + \beta)}$, $q = {1 \over 2} {\lambda (\lambda - \alpha) \over \alpha}$, brings system~\eqref{e6} into the  generic symmetric quadratic dynamical system~\eqref{e8} 
with $\widetilde{\gamma} = {1 \over 4} {\beta + 4 \gamma \over \alpha + \beta}$.
For $\alpha \ne 4 \gamma$ the non-degenerate matrix $\mathcal{B}(\lambda, q)$ with $\lambda = \sqrt{\alpha (\alpha - 4 \gamma)}$, $q = {1 \over 2} {\lambda (\lambda - \alpha) \over \alpha}$, brings system \eqref{e6} into the  generic symmetric quadratic dynamical system \eqref{e9} 
with $\widetilde{\beta} = {\beta + 4 \gamma \over \alpha - 4 \gamma}$.
For $\alpha = - \beta = 4 \gamma$ the non-degenerate matrices $\mathcal{B}(\lambda, q)$ with $\lambda \ne 0$, $q = {1 \over 2} {\lambda (\lambda - \alpha) \over \alpha}$, bring system \eqref{e6} into the  generic symmetric quadratic dynamical system~\eqref{e10}.

For $n \geqslant 3$ the non-degenerate matrix $\mathcal{B}(\lambda, q)$ with $\lambda = {\alpha \over \widetilde{\alpha}}$, $q = - {\alpha (\delta \widetilde{\alpha} - \widetilde{\delta} \alpha) \over \widetilde{\alpha}^2 (\alpha + n \delta)}$,  $\widetilde{\alpha} \ne 0$, $\widetilde{\alpha} + n \widetilde{\delta} \ne 0$ brings system~\eqref{e7} into the  generic symmetric quadratic dynamical system
\[
\xi_k' = \widetilde{\alpha} \xi_k^2 + {\widetilde{\alpha} (2 \alpha \delta \widetilde{\alpha} -2 \widetilde{\delta} \alpha^2+ \beta \widetilde{\alpha} \alpha + \beta \widetilde{\alpha} n \delta)\over \alpha^2 (\widetilde{\alpha} + n \widetilde{\delta})} \xi_k p_1(\xi) + {(\delta^2 \widetilde{\alpha}^2 - \widetilde{\delta}^2 \alpha^2+ \gamma \widetilde{\alpha}^2 \alpha + \gamma \widetilde{\alpha}^2 n \delta)\over \alpha^2 (\widetilde{\alpha}+n \widetilde{\delta})} p_1(\xi)^2 + \widetilde{\delta} p_2(\xi),
\]
$k = 1, \dots, n$. Thus in each orbit one can find a system with $\widetilde{\alpha} = 1$ and $\widetilde{\delta} = 0$. We get the system \eqref{e11} 
with~$\widetilde{\beta} = {(2 \alpha \delta + \beta \alpha + \beta n \delta)\over \alpha^2}$, $\widetilde{\gamma} = {(\delta^2 + \gamma \alpha + \gamma n \delta)\over \alpha^2}$.
\end{proof}

	\section{Algebraic integrability.}

	\subsection{Algebraic integrability of symmetric quadratic dynamical systems.}\text{}

In this section we introduce the notion of algebraic integrability.

	Consider the equation
\begin{equation} \label{e12}
\xi^n - h_1 \xi^{n-1} + \ldots + (-1)^n h_n = 0
\end{equation}
with the roots $(\xi_1, \dots, \xi_n)$. By Vieta's formulas we have $h_k = \sigma_k(\xi)$, $k = 1, \dots, n$, therefore $(h_1, \dots, h_n) = S (\xi_1, \dots, \xi_n)$, where $S$ is the universal algebraic map \eqref{S}. 
The discriminant of \eqref{e12} is $\Delta_n = \prod_{i<j} (\xi_i - \xi_j)^2$. It is a symmetric polynomial of $\xi_1, \dots, \xi_n$, we have $\deg \Delta_n = - 4 n (n-1)$. Denote by $\widehat{S}$ the restriction of the universal algebraic map on $\{ \xi \colon \Delta_n \ne 0\}$. The map $\widehat{S}$ is a covering. Therefore for each set of functions $(h_1(t), \dots, h_n(t)) \subset Im \widehat{S}$ the set of functions $(\xi_1(t), \dots, \xi_n(t)) \subset \widehat{S}^{-1}(h_1(t), \dots, h_n(t))$ is uniquely defined by initial data.

\begin{lem}
For a symmetric quadratic dynamical system of the form \eqref{e2} one has
\[
\mathcal{L} \Delta_n = (n-1) \left(2 \alpha + n \beta\right) \sigma_1(\xi) \Delta_n.
\]
\end{lem}

\begin{proof}
We have
\begin{multline*}
{\mathcal{L} \Delta_n \over \Delta_n} = \sum_{i<j} {2 \mathcal{L} (\xi_i - \xi_j) \over (\xi_i - \xi_j)} = 
\sum_{i<j} \sum_k {2 (\alpha \xi_k^2 + \beta \xi_k p_1(\xi) + \gamma p_1(\xi)^2 + \delta p_2(\xi)) \over (\xi_i - \xi_j)} {\partial (\xi_i - \xi_j) \over \partial \xi_k} = \\
= 2 \sum_{i<j} {(\alpha \xi_i^2 + \beta \xi_i p_1(\xi)) - (\alpha \xi_j^2 + \beta \xi_j p_1(\xi)) \over (\xi_i - \xi_j)} = 2 \sum_{i<j} \left( \alpha (\xi_i + \xi_j) + \beta p_1(\xi) \right) = (n-1) \left (2 \alpha + n  \beta\right) p_1(\xi).
\end{multline*}
The remark $p_1(\xi) = \sigma_1(\xi)$ finishes the proof.
\end{proof}

This lemma implies each symmetric quadratic dynamical system defines a vector field $\mathcal{L}$ tangent to the discriminant curve $\{\xi: \Delta_n = 0 \}$.
Therefore for each symmetric quadratic dynamical system its restriction on the space $\{\xi: \Delta_n \ne 0 \}$ is well-defined.

\begin{dfn}
We say that a symmetric system \eqref{e1'} with initial initial data $(\xi_1(t_0), \dots, \xi_n(t_0)) \notin \Delta_n$ is \emph{algebraically integrable by a set of functions} $\left(h_1(t), \dots, h_n(t)\right)$ if $S(\xi_1(t_0), \dots, \xi_n(t_0))=(h_1(t_0), \dots, h_n(t_0))$, the set $\left(h_1(t), \dots, h_n(t)\right)$ lies in the image of $\widehat{S}$ for any $t$ and the uniquely-defined set $(\xi_1(t), \dots, \xi_n(t)) \subset \widehat{S}^{-1}\left(h_1(t), \dots, h_n(t)\right)$ is a solution of the dynamical system \eqref{e1'}.
\end{dfn}

\begin{problem}[Algebraic integrability] \label{prbl}
Given a symmetric system of the form \eqref{e1'}, find an ordinary differential equation of order $n$ and differential polynomials $h_2, \dots, h_n$ in $h_1$ such that for the each solution $h_1$ of the ordinary differential equation each set of functions $h_1, h_2, \dots, h_n$ such that $\Delta(h_1, \dots, h_n) \ne 0$ algebraically integrates this system.\end{problem}

\begin{thm} \label{t50}
For each generic symmetric system \eqref{e1'} with the initial data $\xi(t_0) = (\xi_1(t_0), \dots, \xi_n(t_0))$ such that $\xi_i(t_0) \ne \xi_j(t_0)$ for any $i \ne j$ there is a solution of the problem of algebraic integrability.
\end{thm}
 
\begin{proof}
Consider system \eqref{e4} with $a_k = \sigma_k$ being the basis of elementary symmetric functions. Under the conditions of the Theorem we have $c_k \ne 0$, $k = 1, \dots, n-1$. Hence, $\sigma_{j}(t)$ with $j = 2, \dots, n$ can be expressed as polynomials in $\sigma_1(t), \dots, \sigma_{j-1}(t)$ and their derivatives from the $(j-1)$-th equation. Thus, the last equation gives a single homogeneous differential equation for $\sigma_1(t)$ in which the coefficient multiplying the highest derivative of $\sigma_1(t)$ is a non-zero constant. Therefore, $(\sigma_1(t), \dots, \sigma_n(t))$ algebraically integrates system \eqref{e4}. The initial conditions in the equation for $\sigma_1(t)$ are as follows: $\sigma_1(t_0) = \sigma_1(\xi_1(t_0), \dots, \xi_n(t_0))= \xi_1(t_0) + \ldots + \xi_n(t_0)$ and 
$\sigma_1^{(k)}(t_0) = (\mathcal{L}^k \sigma_1)(t_0)$. The condition $\xi_i(t_0) \ne \xi_j(t_0)$ for any $i \ne j$ guarantees $\Delta_n \ne 0$.
\end{proof}

Thus, we have reduced the problem of
integrability of a generic symmetric quadratic dynamical system to the question of solving an ordinary differential
equation in $h = \sigma_1$. Set $\zeta = (h, \dots, h^{(n-1)})^\top$. This equation takes the form
\begin{equation} \label{e13}
h^{(n)} + \sum_{||\omega|| = -4 (n+1)} \lambda_{\omega} \zeta^{\omega} = 0
\end{equation}
for the multi-index $\omega = (i_1, \dots, i_n)$ with $||\omega|| = i_k \deg h^{(k-1)}$ and the grading $\deg h = - 4$, $\deg t = 4$.
For relations of such equations to heat equation solutions see \cite{FA}.

\begin{thm} \label{t51}
Each almost generic symmetric system \eqref{e1'} with the initial data $\xi(t_0) = (\xi_1(t_0), \dots, \xi_n(t_0))$ such that $\xi_i(t_0) \ne \xi_j(t_0)$ for any $i \ne j$ is algebraically integrable, namely there is an ordinary differential equation of order $n-1$, differential polynomials $h_2, \dots, h_{n-1}$ in $h_1$, and a first-order ordinary differential equation on $h_n$ with coefficients being differential polynomials in $h_1$ such that for the each solution $h_1$ of the ordinary differential equation and each corresponding solution $h_n$ the set of functions $h_1, h_2, \dots, h_n$ with appropriate initial conditions algebraically integrates this system.
\end{thm}
 
\begin{proof} According to Theorems \ref{tas}, \ref{t11} and \ref{t50}, it is sufficient to prove the Theorem for $n=2$ and $3$ in the case of almost generic non-generic systems.
Following the proof of theorem \ref{t50}, consider system \eqref{e4} with $a_k = \sigma_k$ being the basis of elementary symmetric functions. Under the conditions of the Theorem we have $c_k \ne 0$ for $k = 1, \dots, n-2$ and $c_{n-1} = 0$, thus $c_1 = 0$ for $n=2$, and $c_1 \ne 0$, $c_2 = 0$ for $n=3$. Hence in the case $n=3$ the function $\sigma_2(t)$ can be expressed as a polynomial in $\sigma_1(t)$ and its derivatives. Thus the $(n-1)$-st equation gives a single homogeneous differential equation for $\sigma_1(t)$ in which the coefficient multiplying the highest derivative of $\sigma_1(t)$ is a non-zero constant. The $n$-th equation takes the form
\[
\sigma_n' = c \sigma_1 \sigma_n + g(\sigma_1, \sigma_1', \dots, \sigma_1^{(n)})
\]
for some constant $c$ and some polynomial $g$. If $c=0$ we find $\sigma_n$ as an integral of $g(\sigma_1, \sigma_1', \dots, \sigma_1^{(n)})$, else set $\sigma_n = \phi_1 \phi_2$, where $(\phi_1, \phi_2)$ is a solution of the system
\begin{align*}
\phi_1'&= c \sigma_1 \phi_1,\\
\phi_2' &=  {1 \over \phi_1} g(\sigma_1, \sigma_1', \dots, \sigma_1^{(n)}).
\end{align*}
Therefore, we have described the set $(\sigma_1(t), \dots, \sigma_n(t))$ that algebraically integrates system \eqref{e4}. The initial conditions in the equations are the same as in theorem \ref{t50}.
\end{proof}

	\subsection{The action of the normalizer subgroup in $\mathrm{GL}(n, \mathbb{C})$ of generic symmetric quadratic dynamical systems on the differential equation related to the problem of integrability.} \text{}
	
	The action of the normalizer subgroup in $\mathrm{GL}(n, \mathbb{C})$ of generic symmetric quadratic dynamical systems on equation \eqref{e13} is induced from the representation of $\mathrm{GL}(n, \mathbb{C})$ on dynamical systems in the following way:
	the matrix $B \in \mathrm{GL}(n, \mathbb{C})$ brings $\xi$ to $\eta = B \xi$, thus the equation on $h = \sigma_1(\xi)$ to the corresponding equation on $\widetilde{h} = \sum_i B_i^j \xi_j$.

	The action of $\mathrm{S}_n$ on equation \eqref{e13} is trivial. Thus the action of the normalizer subgroup in $\mathrm{GL}(n, \mathbb{C})$ of generic symmetric quadratic dynamical systems coincides with the action of $\mathcal{B}$. For the matrix $\mathcal{B}(\lambda, q)$ we have $\widetilde{h} = (\lambda + n q) h$.
Thus equation \eqref{e13} becomes
\[
\widetilde{h}^{(n)} + \sum_{||\omega|| = -4 (n+1)} \widetilde{\lambda}_{\omega} \widetilde{\zeta}^{\omega} = 0
\]
for $\widetilde{\zeta} = (\widetilde{h}, \dots, \widetilde{h}^{(n-1)})^\top$, $\widetilde{\lambda}_{\omega} = (\lambda + n q)^{1 - |\omega|} \lambda_{\omega}$, $\omega = (i_1, \dots, i_n)$ and $|\omega| = \sum_k i_k$. The other symmetric functions in this case are expressed as functions of $\sigma_k(\xi)$ as
\[
\sigma_k(\eta) = \sum_{m=0}^k \binom{n-k+m}{m} \lambda^{k-m} q^m \sigma_{k-m}(\xi) \sigma_1(\xi)^m,
\]
where $\sigma_0(\xi) = 1$. This formulas are proved by direct calculations using the definition of $\sigma_k$.

	\section{Applications.}

In this section we show how our method of algebraic integration works for finding solutions of known problems
of quadratic dynamical systems theory.

	\subsection{One-dimensional quadratic dynamical systems.}\text{}

Each one-dimensional quadratic dynamical system has the form
\[
\xi_1' = \alpha \xi_1^2
\]
for some $\alpha$. It is symmetric and generic for any $\alpha$ according to Theorem \ref{t11}. We have $\sigma_1(\xi) = p_1(\xi) = \xi_1$, thus system \eqref{e3} for this generators as well as equation \eqref{e13} coincides with the dynamical system considered. 

For $\alpha \ne 0$ the general solution to this system is
\[
\xi_1 = {1 \over c - \alpha t}
\]
for some constant $c$. For $\alpha = 0$ the general solution is $\xi_1 = c$. 

The one-dimensional quadratic dynamical system is algebraically integrable by $\xi_1(t)$ given above.

	\subsection{Two-dimensional symmetric quadratic dynamical systems.}\text{}

The general two-dimensional symmetric quadratic dynamical system has the form
\begin{align} \label{e14}
\xi_1' &= (\alpha + \beta + \gamma) \xi_1^2 + (\beta + 2 \gamma) \xi_1 \xi_2 + \gamma \xi_2^2, \\
\xi_2' &= (\alpha + \beta + \gamma) \xi_2^2 + (\beta + 2 \gamma) \xi_1 \xi_2 + \gamma \xi_1^2. \nonumber
\end{align}
It is generic for $\alpha \ne 0$.
The orbits of the action of $\mathcal{B}(\lambda, q)$ on such systems are described in Lemma \ref{lem16}.

The dynamical system \eqref{e3} for homogeneous generators $a_1 = p_1$ and $a_2 = p_2$ takes the form
\begin{align*}
p_1' &= (\beta + 2 \gamma) p_1^2 + \alpha p_2, \\
p_2' &= - (\alpha - 2 \gamma) p_1^3 + (3 \alpha + 2 \beta) p_1 p_2.
\end{align*}
For homogeneous generators $a_1 = \sigma_1$ and $a_2 = \sigma_2$ it takes the form
\begin{align*}
\sigma_1' &= (\alpha + \beta + 2 \gamma) \sigma_1^2 - 2 \alpha \sigma_2, \\
\sigma_2' &= \gamma \sigma_1^3 + (\alpha + 2 \beta) \sigma_1 \sigma_2.
\end{align*}

Thus, for $\alpha \ne 0$ the function $\sigma_1 = p_1$ is the solution of equation \eqref{e13} which takes the form
\begin{equation} \label{e15}
h'' - (3 \alpha + 4 \beta + 4 \gamma) h h' + (\alpha + \beta)  (\alpha + 2 \beta + 4 \gamma) h^3 = 0
\end{equation}
and we have
\begin{equation} \label{e16}
\sigma_2 = {\alpha + \beta + 2 \gamma \over 2 \alpha} \sigma_1^2 - {1 \over 2 \alpha} \sigma_1'.
\end{equation}

In this case equation \eqref{e12} for $h_i = \sigma_i$ takes the form
\[
\xi^2 - \sigma_1 \xi + \sigma_2 = 0.
\]
We have $\Delta_2 = \sigma_1^2 - 4 \sigma_2$. Therefore $\Delta_2 = 0$ if $\sigma_1(t)$ satisfies the differential equation
\[
2 h' = (\alpha + 2 \beta + 4 \gamma) h^2.
\]

The general two-dimensional generic symmetric quadratic dynamical system \eqref{e14} is algebraically integrable by the set $\left(\sigma_1(t), \sigma_2(t)\right)$ where $\sigma_1(t) = h(t)$ is a solution to \eqref{e15} such that $
2 \sigma_1'(t_0) \ne (\alpha + 2 \beta + 4 \gamma) \sigma_1(t_0)^2$ and $\sigma_2(t)$ is defined by \eqref{e16}.

Let us consider special cases of equation \eqref{e15}. Denote $\lambda_1 = - (3 \alpha + 4 \beta + 4 \gamma)$, $\lambda_2 = (\alpha + \beta)  (\alpha + 2 \beta + 4 \gamma)$, so that equation \eqref{e15} takes the form
\begin{equation} \label{e17}
h'' + \lambda_1 h h' + \lambda_2 h^3 = 0.
\end{equation}

In the special case $\lambda_1 = 0$ the general solution to \eqref{e17} is
$h(t) = k_2 \mathrm{sn}\left(\left(\sqrt{{\lambda_2 \over 2}} t+ k_1\right) k_2, i\right)$, where $\mathrm{sn}$ is the Jacobi sine, $i$ is the square root of $-1$ and $k_1, k_2$ are constants.

In the special case $\lambda_2 = 0$ the general solution to \eqref{e17} is
$h(t) = {\sqrt{2 k_1}  \over \sqrt{\lambda_1}} \tanh\left( \sqrt{{k_1 \lambda_1 \over 2}} (t+ k_2)\right)$, where $k_1, k_2$ are constants.

In the special case $\lambda_1^2 = 9 \lambda_2$ the general solution to \eqref{e17} is
$h(t) = 6 (k_1 t + k_2) / \lambda_1 (k_1 t^2 + 2 k_2 t + 2)$, where $k_1, k_2$ are constants.

Let us now consider the case of non-generic two-dimensional symmetric quadratic dynamical system \eqref{e14}, that is the case $\alpha = 0$. For homogeneous generators $a_1 = \sigma_1$ and $a_2 = \sigma_2$ system \eqref{e3} takes the form
\begin{align*}
\sigma_1' &= (\beta + 2 \gamma) \sigma_1^2, \\
\sigma_2' &= 2 \beta \sigma_1 \sigma_2 + \gamma \sigma_1^3.
\end{align*}

For $\delta = (\beta + 2 \gamma) \ne 0$ the general solution to this system is
\[
\sigma_1 = - \left(\delta t + k_1\right)^{-1}, \qquad \sigma_2 = {1 \over 4} \left(\delta t + k_1\right)^{-2} + k_2 \left(\delta t + k_1\right)^{-2 \beta \over \delta}
\]
for some constants $k_1$, $k_2$.

For $(\beta + 2 \gamma) = 0$ the general solution to this system is
\[
\sigma_1 = k_1, \qquad \sigma_2 = k_2 e^{2 b k_1 t} + {1 \over 4} k_1^2
\]
for some constants $k_1$, $k_2$.

Therefore in this case we get the following explicit result:

The general two-dimensional non-generic symmetric quadratic dynamical system \eqref{e14} is algebraically integrable in $U = \mathbb{C}$ by the set $\left(\sigma_1(t), \sigma_2(t)\right)$ where $\sigma_1$ and $\sigma_2$ with $k_2 = 0$ are given above.

\subsection{Three-dimensional symmetric quadratic dynamical systems.}\text{}

The general three-dimensional symmetric quadratic dynamical system 
has the form
\begin{align}
\xi_1' &= (\alpha + \beta + \gamma + \delta) \xi_1^2 + (\beta + 2 \gamma) \xi_1 (\xi_2 + \xi_3) + (\gamma + \delta) (\xi_2^2 + \xi_3^2) + 2 \gamma \xi_2 \xi_3, \nonumber \\
\xi_2' &= (\alpha + \beta + \gamma + \delta) \xi_2^2 + (\beta + 2 \gamma) \xi_2 (\xi_1 + \xi_3) + (\gamma + \delta) (\xi_1^2 + \xi_3^2) + 2 \gamma \xi_1 \xi_3,\label{sys3}\\
\xi_3' &= (\alpha + \beta + \gamma + \delta) \xi_3^2 + (\beta + 2 \gamma) \xi_3 (\xi_1 + \xi_2) + (\gamma + \delta) (\xi_1^2 + \xi_2^2) + 2 \gamma \xi_1 \xi_2. \nonumber 
\end{align}
It is almost generic for $\alpha + 3 \delta \ne 0$ and generic for $\alpha \ne 0$ and $\alpha + 3 \delta \ne 0$.

Equation \eqref{e12} takes the form
\[
\xi^3 - h_1 \xi^2 + h_2 \xi - h_3 = 0
\]
with $\Delta_3 = - 27 h_3^2 + 18 h_1 h_2 h_3 - 4 h_1^3 h_3 - 4 h_2^3 + h_1^2 h_2^2$.

The dynamical system \eqref{e3} for homogeneous generators $a_1 = \sigma_1$, $a_2 = \sigma_2$, $a_3 = \sigma_3$ takes the form
\begin{align*}
\sigma_1' &= (\alpha + \beta + 3 \gamma + 3 \delta) \sigma_1^2 - 2 (\alpha + 3 \delta) \sigma_2, \\
\sigma_2' &= 2 (\gamma + \delta) \sigma_1^3 + (\alpha + 2 \beta - 4 \delta) \sigma_1 \sigma_2 - 3 \alpha \sigma_3,\\
\sigma_3' &= (\gamma + \delta) \sigma_2 \sigma_1^2 - 2 \delta \sigma_2^2 + (\alpha + 3 \beta) \sigma_3 \sigma_1.
\end{align*}

The general three-dimensional generic symmetric quadratic dynamical system \eqref{sys3} is algebraically integrable by the set $\left(\sigma_1(t), \sigma_2(t), \sigma_3(t)\right)$ 
where $\sigma_1$ is a solution of equation
\begin{equation} \label{lh}
\lambda_1 h''' + \lambda_2 h'' h + \lambda_3 (h')^2 + \lambda_4 h' h^2 + \lambda_5 h^4 = 0
\end{equation}
with
\begin{align*}
\lambda_1 &= \alpha + 3 \delta,\\
\lambda_2 &= - \lambda_1 (4 \alpha + 7 \beta + 6 \gamma + 2 \delta),\\
\lambda_3 &= - (3 \alpha^2 + 6 \delta^2) - 2 (2 \alpha \beta + 3 \alpha \gamma + 4 \alpha \delta + 6 \beta \delta + 9 \gamma \delta),\\
\lambda_4 &= 6 \alpha^3 + 2 (11 \alpha^2 \beta + 9 \beta^2 \alpha + 15 \alpha^2 \gamma + 13 \alpha^2 \delta + 12 \alpha \delta^2 + 27 \beta^2 \delta + 18 \beta \delta^2) + 36 (\alpha \beta \gamma + 2 \alpha \beta \delta + 2 \alpha \gamma \delta + 3 \beta \gamma \delta), \\
\lambda_5 &= - (\alpha + 3 \beta + 9 \gamma + 3 \delta) (\alpha^3 + 3 \alpha^2 \beta + 2 \alpha \beta^2 + \alpha^2 \gamma + 3 \alpha^2 \delta + 6 \beta^2 \delta + 8 \alpha \beta \delta)
\end{align*}
and
\begin{align*}
- 2 (\alpha + 3 \delta) \sigma_2 &= \sigma_1' - (\alpha + \beta + 3 \gamma + 3 \delta) \sigma_1^2,\\
6 \alpha (\alpha + 3 \delta) \sigma_3 &= \sigma_1'' - (3 \alpha + 4 \beta + 6 \gamma + 2 \delta) \sigma_1' \sigma_1 + (\alpha^2 + 2 \beta^2 + 3 \alpha \beta + 7 \alpha \gamma + 3 \alpha \delta + 6 \beta \gamma + 2 \beta \delta) \sigma_1^3,
\end{align*}
where the initial conditions satisfy
$
27 \sigma_3(t_0)^2 \ne 18 \sigma_1(t_0) \sigma_2(t_0) \sigma_3(t_0) - 4 \sigma_1(t_0)^3 \sigma_3(t_0) - 4 \sigma_2(t_0)^3 + \sigma_1(t_0)^2 \sigma_2(t_0)^2.
$

Let us now consider the case of almost generic non-generic three-dimensional symmetric quadratic dynamical system \eqref{sys3},
that is the case $\alpha = 0$, $\delta \ne 0$.
For homogeneous generators $a_1 = \sigma_1$, $a_2 = \sigma_2$, $a_3 = \sigma_3$ system \eqref{e3} takes the form
\begin{align*}
\sigma_1' &= (\beta + 3 \gamma + 3 \delta) \sigma_1^2 - 6 \delta \sigma_2, \\
\sigma_2' &= 2 (\gamma + \delta) \sigma_1^3 + 2(\beta - 2 \delta) \sigma_1 \sigma_2,\\
\sigma_3' &= (\gamma + \delta) \sigma_2 \sigma_1^2 - 2 \delta \sigma_2^2 + 3 \beta \sigma_3 \sigma_1.
\end{align*}

The almost generic non-generic three-dimensional symmetric quadratic dynamical system \eqref{sys3} is algebraically integrable by the set $\left(\sigma_1(t), \sigma_2(t), \sigma_3(t)\right)$ 
where $\sigma_1$ is a solution of equation
\[
h''(t) - 2 (2 \beta + 3 \gamma + \delta) h(t) h'(t) + 2 \beta (\beta + 3 \gamma + \delta) h(t)^3 = 0,
\]
we have $\sigma_2 = {(\beta + 3 \gamma + 3 \delta) \over 6 \delta} \sigma_1^2 - {1 \over 6 \delta} \sigma_1'$, and for this $\sigma_1$ the function $\sigma_3$ is a solution of equation
\[
18 \delta \sigma_3' -54 \beta \delta \sigma_3 \sigma_1 + (\sigma_1')^2 - (2 \beta + 3 \gamma + 3 \delta) \sigma_1^2 \sigma_1' + \beta (\beta +3 \gamma +3 \delta) \sigma_1^4 = 0.
\]
The equation on initial conditions is the same as in the previous case.

\subsection{Lotka-Volterra type system.}\text{}

Let us consider the dynamical system
\begin{equation} \label{kp}
\xi_k' = \xi_k \left(\sum_{l=1}^3 \xi_l - 2 \xi_k \right), \qquad k = 1, 2, 3.
\end{equation}
It is symmetric and has the form \eqref{sys3} for $\alpha = - 2$, $\beta = 1$, $\gamma = 0$, $\delta = 0$, thus it is generic.

This system is algebraically integrable by the set of functions $\left(\sigma_1(t), \sigma_2(t), \sigma_3(t)\right)$ where $\sigma_1$ is a solution of equation (see \eqref{lh})
\begin{equation} \label{1122}
h''' + h'' h + 2 (h')^2 - 2 h' h^2 = 0
\end{equation}
and $4 \sigma_2 = \sigma_1' + \sigma_1^2$, $- 6 4 \sigma_3 = \sigma_1'' + 2 \sigma_1' \sigma_1$.

S. Kovalevskaya introduced system \eqref{kp} in a letter to G.~Mittag-Leffler. She obtained the result that this system is integrable in elliptic functions.
It has the solution $\xi_1 = (\ln \mathrm{sn}(t))'$, $\xi_2 = (\ln cn(t))'$, $\xi_3 = (\ln dn(t))'$.
We have $\sigma_1 = \xi_1 + \xi_2 + \xi_3$, therefore the elliptic function $(\ln((\mathrm{sn}(t)^2)'))'$ is a solution of \eqref{1122}.

For dimensions $n \geqslant 3$ a generalization of system \eqref{kp} is a Lotka-Volterra type system (see \cite{BM})
\begin{equation} \label{kpn}
\xi_k' = \xi_k \left(\sum_{l=1}^n \xi_l - 2 \xi_k \right), \qquad k = 1, \dots, n.
\end{equation}
This system is symmetric and generic for any $n$, thus our general method gives the answer to the question of algebraical integrability for this system.

\begin{ex} 
For $n = 4$ equation \eqref{e13} takes the form
\[
h'''' - h''' h + 5 h'' h' - 4 h'' h^2 - 8 (h')^2 h + 4 h' h^3 = 0
\]
and the differential polynomials are
\begin{equation} \label{hhh}
\sigma_2 = {1 \over 4} (\sigma_1' + \sigma_1^2), \quad \sigma_3 = {1 \over 24} (\sigma_1'' + 2 \sigma_1' \sigma_1), \quad \sigma_4 = {1 \over 192} (\sigma_1''' + \sigma_1'' \sigma_1 + 2 (\sigma_1')^2 - 2 \sigma_1' \sigma_1^2).
\end{equation}
\end{ex} 

\begin{lem}
The system \eqref{kpn} for $n = 2$ has only one quadratic integral $x_1 x_2$.\\
For $n = 3$ it has two independent quadratic integrals
\[
a x_1 x_2 + b x_1 x_3 + c x_2 x_3, \qquad a + b + c = 0.
\]
For $n = 4$ it has two independent quadratic integrals
\[
a (x_1 x_2 + x_3 x_4) + b (x_1 x_3 + x_2 x_4) + c (x_2 x_3 + x_1 x_4), \qquad a + b + c = 0.
\]
For $n \geqslant 5$ the system \eqref{kpn} has no quadratic integrals.
\end{lem}

\begin{proof}
System \eqref{kpn} is such that $\mathcal{L} P(x_1, \dots, x_{n-1}, 0) = \left.\left(\mathcal{L} P(x_1, \dots, x_{n-1}, x_n)\right)_{}\right|_{x_n = 0}$. Therefore for an integral $P(x_1, \dots, x_{n-1}, x_n)$ of system \eqref{kpn} for some $n$
the polynomial $P(x_1, \dots, x_{n-1}, 0)$ is an integral of system \eqref{kpn} for $n-1$.

One can check directly that the lemma holds for $n = 4$ and $n = 5$, all other cases becoming corollaries of the last statement.
\end{proof}

Another approach to integration of system \eqref{kpn} (see \cite{ArX}) is a relation of this system 
to a generalization of the Euler top
\[
\eta_k' = \prod_{i \ne k} \eta_i, \qquad k = 1, \dots, n,
\]
using the transform $\xi_k = {1 \over \eta_k} \prod_{i \ne k} \eta_i$.
One gets $n-1$ functionally independant rational integrals of the form
\[
P_{i,j} = \left( {(\xi_i - \xi_j) \over \xi_i \xi_j} \right)^{n-2}  \prod_{k=1}^n \xi_k.
\]
 
\begin{rem}
The system
\begin{equation} \label{kp2}
\xi_k' = \xi_k \left(\sum_{l=1}^3 \xi_l - m \xi_k \right), \quad k = 1, 2, 3, 
\end{equation}
for $m \ne 2$ has no quadratic integrals.
\end{rem}

\subsection{The Darboux-Halphen system.}\text{}

Let us consider the classical Darboux-Halphen system
\begin{align*}
\xi_1' = \xi_2 \xi_3 - \xi_3 \xi_1 - \xi_1 \xi_2, \\
\xi_2' = \xi_3 \xi_1 - \xi_1 \xi_2 - \xi_2 \xi_3, \\
\xi_3' = \xi_1 \xi_2 - \xi_2 \xi_3 - \xi_3 \xi_1.
\end{align*}
It is symmetric and has the form \eqref{sys3} for $\alpha = 2$, $\beta = - 2$, $\gamma = {1 \over 2}$, $\delta = - {1 \over 2}$, thus it is generic.

The Darboux-Halphen system is algebraically integrable by the set of functions $\left(\sigma_1(t), \sigma_2(t), \sigma_3(t)\right)$ where $\sigma_1$ is a solution of equation (see \eqref{lh})
\[
h''' + 4 h'' h - 6 (h')^2 = 0
\]
and
$\sigma_2= - \sigma_1'$,
$6 \sigma_3 = \sigma_1''$.
Therefore we have $\sigma_1 = - {1 \over 2} y$, where $y(t)$ is a solution of the famous Chazy-3 equation
\[
y''' - 2 y'' y + 3 (y')^2 = 0.
\]

Let us consider a generalization of the Darboux-Halphen system, namely
\begin{align*}
\xi_1' = a (\xi_2 \xi_3 - \xi_3 \xi_1 - \xi_1 \xi_2) + b \xi_1^2, \\
\xi_2' = a (\xi_3 \xi_1 - \xi_1 \xi_2 - \xi_2 \xi_3) + b \xi_2^2, \\
\xi_3' = a (\xi_1 \xi_2 - \xi_2 \xi_3 - \xi_3 \xi_1) + b \xi_3^2.
\end{align*}
It is symmetric and has the form \eqref{sys3} for $\alpha = 2 a + b$, $\beta = - 2 a$, $\gamma = {1 \over 2} a$, $\delta = - {1 \over 2} a$, thus it is generic for $2 a \ne - b$, $a \ne - 2 b$. It coincides with the classical Darboux-Halphen system for $a = 1$, $b = 0$.

According to \eqref{lh}, for $a \ne b$ for this system we have $2 (a - b) \sigma_1 = - y$, where $y(t)$ is a solution of
 \begin{equation} \label{D2}
y''' - 2 y y'' + 3 (y')^2 + c (6 y' - y^2)^2 = 0, \qquad c = {b^2 \over 4 (a + 2 b) (a - b)}.
\end{equation}
This non-linear ordinary differential equation coincides with the Chazy-3 equation for $c = 0$ and with the Chazy-12 equation for $c = {4 \over k^2 - 36}$, $k \in \mathbb{N}$, $k > 1$, $k \ne 6$. These equations are examples of known third-order differential equations with the Painlev\'e property.

\subsection{An almost generic classical system.} \text{}

Let us consider the classical three-dimensional dynamical system in Lax form
\begin{equation} \label{qsds}
L' = [[K,L^k], L],
\end{equation}
where $L = \begin{pmatrix} 0 & \xi_1 & \xi_2 \\ - \xi_1 & 0 & \xi_3 \\ - \xi_2 & - \xi_3 & 0 \end{pmatrix}$ and $K$ is a $3 \times 3$ constant matrix.
This system is quadratic for $k = 1$. In this case for the equality of the right and the left part in \eqref{qsds} it is necessary that
$K = \begin{pmatrix} d & m_1 & m_2 \\ - m_1 & d & m_3 \\ - m_2 & - m_3 & d \end{pmatrix}$ for some $d, m_1, m_2, m_3$. Moreover, the resulting system does not depend on $d$, thus we can put $d=0$.
The system is symmetric if and only if $m_1 = m_2 = m_3 = m$. Therefore any symmetric quadratic dynamical system of the form \eqref{qsds} up to a rescaling $\xi_i \to {\xi_i \over m}$ coincides with the symmetric quadratic dynamical system
\begin{align}
\xi_1' &= \xi_1 (\xi_2 + \xi_3) - (\xi_2^2 + \xi_3^2),\nonumber \\
\xi_2' &= \xi_2 (\xi_1 + \xi_3) - (\xi_1^2 + \xi_3^2), \label{lastsys} \\
\xi_3' &= \xi_3 (\xi_1 + \xi_2) - (\xi_1^2 + \xi_2^2). \nonumber
\end{align}
This system has the form \eqref{sys3} for $\alpha = 0$, $\beta = 1$, $\gamma = 0$, $\delta = - 1$.
It is not generic, but almost generic.

From the form \eqref{qsds} we get integrals of the system of the form $tr(L^k)$. We have $tr(L^{2k+1}) = 0$ and $tr(L^{2k}) = c_k (\xi_1^2 + \xi_2^2 + \xi_3^2)^k$ for constant $c_k$ with $c_1 = - 2$, thus this method gives only one integral equal $p_2(\xi)$. Let us use our method to algebracally integrate this system.

For system \eqref{lastsys} the funtion
$\sigma_1$ is a solution of equation
\begin{equation} \label{feq}
h'' - 2 h' h = 0
\end{equation}
and
$\sigma_2= {1 \over 6} (\sigma_1' + 2 \sigma_1^2)$. The expression $\sigma_1^2 - 2 \sigma_2 = p_2$ is an integral of the system.
For a solution $\sigma_1$ of \eqref{feq} the function $\sigma_3$ is a solution of
\begin{equation} \label{feq2}
\sigma_3' = 3 \sigma_1 \sigma_3 + {1 \over 18} (\sigma_1' - \sigma_1^2) (\sigma_1' + 2 \sigma_1^2).
\end{equation}

The general solution of \eqref{feq} is $\sigma_1 = c \tan(c t+b)$. We get $\sigma_2 = - {c^2 \over 6} \left(2 - {3 \over \cos(c t + b)^2}\right)$.
The corresponding general solution to \eqref{feq2} is
\[
\sigma_3 = - {c^3 \sin(c t + b) (2 \cos(c t + b)^2 - 5) + k \over 54 \cos(c t + b)^3}
\]
for a constant $k$.

System \eqref{lastsys} is algebraically integrable by the set $\left(\sigma_1(t), \sigma_2(t), \sigma_3(t)\right)$ where $\sigma_1$, $\sigma_2$ and $\sigma_3$ are given above and $2 c^6 + k^2 \ne 0$.

\end{document}